\documentclass[12pt,reqno]{amsart}

\usepackage{amssymb,amsmath}

\numberwithin{equation}{section}

\newtheorem{thm}{Theorem}[section]
\newtheorem{pro}[thm]{Proposition}
\newtheorem{lemma}[thm]{Lemma}
\newtheorem{cor}[thm]{Corollary}

\theoremstyle{definition}
\newtheorem{definition}[thm]{Definition}
\newtheorem{remark}[thm]{Remark}
\newtheorem{example}[thm]{Example}

\DeclareMathOperator*{\QF}{QF}
\DeclareMathOperator*{\codim}{codim}
\DeclareMathOperator*{\Vect}{Vect}
\DeclareMathOperator*{\Sing}{Sing}
\DeclareMathOperator*{\Supp}{Supp}
\DeclareMathOperator*{\degree}{degree}
\DeclareMathOperator*{\tor}{tor}
\DeclareMathOperator*{\Hom}{Hom}

\DeclareMathOperator*{\BL}{BL}
\newcommand{\PP}{\mathbb{P}}
\newcommand{\cO}{\mathcal{O}}
\newcommand{\scrF}{\mathcal{F}}
\newcommand{\scrV}{\mathcal{V}}

\newcommand{\Q}{\mathbb{Q}}

\newcommand{\scrp}{\mathfrak{p}} 
\newcommand{\scrq}{\mathfrak{q}} 

\begin{document}

\title[Bertini type results and applications]{Bertini type results and their applications} 

\author[I. Biswas]{Indranil Biswas}

\address{Department of Mathematics, Shiv Nadar University, NH91, Tehsil
Dadri, Greater Noida, Uttar Pradesh 201314, India}

\email{indranil.biswas@snu.edu.in, indranil29@gmail.com}

\author[M. Kumar]{Manish Kumar}

\address{Statistics and Mathematics Unit, Indian Statistical Institute,
Bangalore 560059, India}

\email{manish@isibang.ac.in}

\author[A.J. Parameswaran]{A. J. Parameswaran}

\address{School of Mathematics, Tata Institute of Fundamental
Research, Homi Bhabha Road, Bombay 400005, India}

\email{param@math.tifr.res.in}



\date{}

\begin{abstract}
We prove Bertini type theorems and give some applications of them. The applications are in the context of 
Lefschetz theorem for Nori fundamental group for normal varieties as well as for geometric formal orbifolds. 
In another application, it is shown that certain class of a smooth quasi-projective variety contains a smooth 
curve such that irreducible lisse $\ell$--adic sheaves on the variety with ``ramification bounded by a branch 
data'' remains irreducible when restricted to the curve.
\end{abstract}

\maketitle

\section{Introduction}

The Bertini theorem asserts that given an embedding of a smooth projective variety $X$ in a projective space 
using the complete linear system of a very ample line bundle on $X$, the intersection of a general hyperplane 
with the image of $X$ is smooth. This is a well known result over fields with arbitrary characteristic. But 
for the more general situation of base point free linear systems, the theorem holds in complete generality 
only when the characteristic of the base field is zero.

We call a dominant morphism $f\,:\, X\, \longrightarrow\, Y$ of varieties over a field $\mathbf k$ generically separable if the induced field extension $k(X)/k(Y)$ is separably generated. For a finite dominant morphism, this condition means that $k(X)/k(Y)$ is a separable field extension. Such morphisms are also called finite separable morphisms. In the case of positive characteristic it is
well-known that even if $f\,:\,X\, \longrightarrow\, Y \subset \PP^N$ is a finite separable dominant morphism with $X$ a smooth
projective variety and $Y$ a projective variety, the pull back of every hyperplane section of $Y$ may be 
singular (see Example \ref{counter-example-normal}). If $f$ happens to be ``residually separable'' then a
Bertini type result holds (see \cite{bertini-CGM}). For Bertini theorem in other contexts, cf. \cite{CP},
\cite{FMZ}.

In this paper we prove Bertini type results under a substantially milder hypothesis. We define 
$A_r\,:=\,A_r(f)$ as the rank $r$ loci of the differential map $df\,:\,TX\, \longrightarrow\, f^*T\PP^N$ and show
that, under a certain hypothesis on these rank loci, the Bertini type results hold (see Theorem 
\ref{main-bertini}). Using this theorem and an inductive argument we obtain the following result.

Let $f\,:\,X\, \longrightarrow\, Y$ be a dominant morphism, where $X$ is of dimension $n$ and
satisfies $R_k$ (i.e., regular in codimension $k$) for some $k\,\le\, n$, $Y\,\subset \,\PP^N$ is a projective variety 
and $H_m$ is a general linear subspace in $\PP^N$ of codimension $m<n$, then the pullback of $H_m$ in $X$ also 
satisfies $R_k$ if certain rank loci are small, more precisely, codimension of $A_{r+j}$ is at least $k-r+1$ 
for all $r\,<\,k$ and $j\,<\,m$ (see Theorem \ref{bertini.linear.subspace}).

As a consequence, we also show that if $f$ is also finite, $X$ and $Y$ are normal and $A_j(F)$ has codimension
at least 2 for $0\,\le\, j\,\le\, n-2$, then the pullback $Y\cap W$ for a general linear subspace $W$ of 
codimension at most $n-1$ is a normal variety (see Lemma \ref{bertini-normal-finite-mor}).

We generalize Lefschetz theorem for Nori fundamental group for normal projective varieties (see Theorem 
\ref{thmnl}). The smooth case was already known \cite{BH}. It should be mentioned that for normal varieties 
only the surjectivity part for the homomorphism between the fundamental groups hold; this is demonstrated by 
an example given in the beginning of Section \ref{sec:lefschetz-normal}.

Lefschetz type theorem for Nori and \'etale fundamental group for a class of projective geometric formal 
orbifolds is also shown (see Theorem \ref{lefschetz-thm}). This is a consequence of the Bertini result and 
Lefschetz theorem for normal varieties. This question was actually raised in \cite{formal.orbifolds}.

Finally, let $X^o$ be a smooth quasiprojective variety and suppose that there exists
a ``super-separable'' $(X,\,P)$
geometric projective formal orbifold (see Definition \ref{super-separable}) containing $X^o$ as open subset 
disjoint with branch locus $\BL(P)$ of $P$. As a consequence of our Bertini result (Lemma 
\ref{bertini-divisor}) we show that there exists a smooth curve $C^o\,\subset\, X^o$ such that the restriction of 
any irreducible lisse $\overline{\Q}_{\ell}$--sheaf on $X^o$ whose ramification is bounded by $P$ to $C^o$ is 
irreducible. This is a variant of a question of Deligne (see \cite[Section 3.2]{esnault}).

\section{Bertini theorem for separable base point free linear systems}

Let $X$ be a projective variety of dimension $n$ defined over an algebraically closed field $\mathbf k$. There
is no assumption on the characteristic of $\mathbf k$. Let $X_{reg}\, \subset\, X$ be the smooth locus of $X$.
Take a dominant morphism
\begin{equation}\label{ef}
f\,:\,X\,\longrightarrow \,Y\,\subset\, {\PP}^N
\end{equation}
onto $Y$. Let
$$
df\,:\,TX\,\longrightarrow\,f^*T\PP^N
$$
be the differential of the map $F$.
The dual projective space parametrizing the hyperplanes in ${\PP}^N$ will
be denoted by $\check{{\PP}}^N$. For each $0\,\leq\, r\,\leq \,n$,
define
\begin{equation}\label{ar}
A_r(f)\,=\,A_r\, :=\, \{x\in X_{reg}~~\mid~~ \{{\rm rank~ of~ f ~at~}~x\}\,:=\,{\rm dim}~df(T_xX)
\,=\,r\} \subset X\, ,
\end{equation}
\begin{equation}\label{br}
B_r\,:=\, \{(x,H)\in A_r\times \check{{\PP}}^N~~\mid~~ df(T_xX)\subset H\} \subset A_r\times \check\PP^N\, ,
\end{equation}
\begin{equation}\label{yr}
Y_r \,:= \,f(A_r)\, .
\end{equation}

Note that the collection $\{A_r\}$ in \eqref{ar}
defines an algebraic stratification of $X$. If $m\,=\, \dim f(X)$ then $A_i$ is empty for all $i\,>\,m$. Moreover, $f$ is generically separable if and only if $A_m$ is open dense subset of $X$. In particular, if the map
$f$ in \eqref{ef} is dominant finite separable morphism, then $A_n$ is an open dense subset of $X$.

For a variety $Z$ of dimension $d$, and an integer $0\, \leq\, k\, \leq\, \dim Z\,=\, d$,
we say that $Z$ satisfies \emph{$R_k$} if $Z$ is regular at all the codimension $k$ points of $Z$.
We note that $Z$ satisfies $R_0$, and it satisfies $R_{d}$ if and only if $Z$ is smooth.

\begin{lemma}\label{lem-a}
Let $H\, \subset\, {\PP}^N$ be a hyperplane.
Let $$f\vert_{f^{-1}(H)}\,:\,f^{-1}(H)\,\longrightarrow\, H$$
be the restriction map. Then $x\,\in\, f^{-1}(H)$ is regular point of $f^{-1}(H)$ if and only if
\begin{itemize}
\item $x\,\in\, X_{reg}$, and

\item $df(T_xX)$ is not contained in $H$.
\end{itemize}
\end{lemma}

\begin{proof}
Let $x\,\in\, X$ be a point. If $x$ is a singular point of $X$, then it is a singular point of $f^{-1}(H)$.

Take a point $x\, \in\, f^{-1}(H)$.
If $x$ is a regular point of $X$, then $x$ is a regular point of $f^{-1}(H)$ if and only if
any element in the local ring ${\mathcal O}_{X,\,x}$ defining
$f^{-1}(H)$ does not lie in the square of the maximal ideal of $\cO_{X,\,x}$.
Denote by $h$ an element in the local ring of ${\PP}^N$ at $f(x)$ defining $H$. Then the function defining
$f^{-1}(H)$ in $X$ is the composition $h\circ f$. By chain rule, the differential of $h\circ f$ is zero
if and only if $df(T_xX)$ is contained in $T_{f(x)}H$. Here by abuse of notation, we identify
the tangent space of any point of a linear space with the linear space itself. Now we see that
$h\circ f$ is in the square of the maximal ideal of ${\mathcal O}_{X,\,x}$ if and only if $df(T_xX)$ is contained in $H$.
This proves the lemma.
\end{proof}

Our first main result is the following.

\begin{thm}\label{main-bertini}
If $X$ satisfies $R_k$ for some $1\,\leq\, k\,<\, n$, then the pullback of a general hyperplane satisfies
$R_k$, provided for each $r\, < \,k$, and each irreducible component $A'_r$ of $A_r$ in \eqref{ar},
\begin{itemize}
\item either $\codim(A'_r)\,>\, k-r$, or

\item $A'_r$ is contained in a fiber of $f$.
\end{itemize}
\end{thm}

\begin{proof}
We denote the component $p^{-1}_1(A'_r)$ of $B_r$ (defined in \eqref{br}) by $B'_r$. The dimension of $A'_r$ will be
denoted by $a_r$.

Notice that $f^{-1}(H)$ is smooth at $x\,\in\, A'_r$ if and only if $(x,\,H)\,\notin\, B'_r$ by Lemma \ref{lem-a}.
%Equivalently, $f$ is transversal to $H$ at $x\,\in\, A'_r$ if and only $(x,\,H)\,\notin\, B'_r$. 
Hence the singular
locus of $f^{-1}(H)$ along $A'_r$ coincides with the fiber of the restriction of the second projection
$p_2\,:\,A'_r\times \check \PP^N \,\longrightarrow\, \check{{\PP}}^N$ to $B'_r\, \subset\, A'_r\times \check \PP^N$, 
in other words,
$$\Sing(f^{-1}(H))\cap A'_r\,=\, \{x\,\in\, A'_r\,\mid\, (x,\,H)\,\in\, B'_r\}\, .$$ Note that the projection
$p_1\, :\, B'_r \,\longrightarrow\, A'_r$ is a smooth fibration of relative dimension $N-r-1$.
Consequently, we have $\dim B'_r\,=\,a_r + N-r-1$.

Consider the image $p_2(B'_r)\, \subset\, \check{{\PP}}^N$
of $B'_r$ under the second projection $p_2$. Now there are the following two possibilities:

Case I:\, $ \dim p_2(B'_r)\, <\, N$. 

Case II:\, $\dim p_2(B'_r)\,=\,N$.

Assume that Case I holds. Then for a general hyperplane $H
\,\in\,\check{{\PP}}^N$ which is not in the closure of
$p_2(B'_r)$, the inverse image $f^{-1}(H)$ is smooth at all points in $A'_r\cap f^{-1}(H)$. Note
that if $A'_r$ is contained in a fiber of $f$, then the restriction $p_2\vert_{B'_r}\, :\,B'_r\, \longrightarrow\,
\check{{\PP}}^N$ is not dominant.

Next assume that Case II holds. Then the general fiber of $p_2\vert_{B'_r}$ has dimension $$a_r+N-r-1-N\,=\, a_r-r-1
\, \geq\, 0\, .$$
This shows that the codimension of the singular locus in
$f^{-1}(H)$ is at least $n-1-(a_r-r-1)\,=\,n-a_r+r \,=\, n-(a_r-r)$. So if $r\,\geq\, k$, then $n-a_r+r
\,>\, k$, because $n-a_r\,>\,0$. Hence $f^{-1}(H)$ always satisfies $R_k$ along the subset $A'_r$
for $r\,\ge\, k$. If $r\,<\,k$, then the statement that $f^{-1}(H)$ is
regular in codimension $k$ at the points of $A'_r\cap f^{-1}(H)$ is equivalent to
the statement that $n-a_r+r \,>\, k$, i.e., $n-a_r \,>\, k-r$.

Consequently, if one of the following two holds:
\begin{itemize}
\item $\codim_X(A'_r)\,> \,k-r$,

\item $A'_r$ is contained in a fiber,
\end{itemize}
then $f^{-1}(H)$ is regular in codimension $k$.
\end{proof}

The following result is an improvement on the existing results in the literature \cite{bertini-Spreafico}, 
\cite{bertini-CGM}.

\begin{cor}\label{residually-separable}
Let $f_r\,:=\,f\vert_{A_r(F)}$, and assume that $f_r\,:\,A_r(f)\,\longrightarrow\, Y_r$
(defined in \eqref{yr}) is generically separable (i.e., the residue field extension induced by $f$ at the generic points of $A_r(f)$ are separably generated field extension) for all $r$.
If $X$ satisfies $R_k$, then $f^{-1}(H)$ satisfies $R_k$ for general $H$. By induction, a general complete
intersection is also regular in codimension $k$.
\end{cor}

\begin{proof}
Under the given hypothesis, we will show that Case II in the proof of Theorem \ref{main-bertini} does not 
occur. Since $f_r\,:\,A_r(f)\,\longrightarrow\, Y_r$ is generically separable, at a general smooth point $x$ of $A_r(f)$
the differential $df_r$ is surjective onto the tangent space of $Y_r$. Hence any hyperplane $H$ that contains 
the image of differential $df(T_x X)$ will contain the image $df_r(T_xA_r(f))$ of the differential of $f_r$,
and consequently it contains the tangent space to $Y_r$ at $f_r(x)$. Therefore, the image of $B_r$ is 
contained in the dual variety of $Y_r$. But $Y_r$ is a proper subvariety of $\PP^N$, and hence its dual 
variety is a proper subvariety of $\check{{\PP}}^N$. Therefore, we conclude that $B_r$ cannot surject onto 
$\check{{\PP}}^N$. This proves that $f^{-1}(H)$ is regular in codimension $k$ along $A_r(f)$.
\end{proof}

\begin{remark}
The difficulty for the converse of Theorem \ref{main-bertini}
stems from the lack of understanding of the behavior of the differential images along $Y_r$. However,
this difficulty is absent when $r\,=\,0$. For $r\,>\,0$, the separability helps as we
know the behavior of tangent space variations from classical Lefschetz Theory.
\end{remark}

For instance only a knowledge of the behavior of the rank zero locus is needed in order to
get information on the $R_1$ property of a general hyperplane section. More precisely, we need
to show that every divisorial component of $A_0$ is contained in a fiber. In 
this case we also have the converse.

\begin{cor}\label{criteria.R_1}
Let $f\,:\,X\,\longrightarrow\, \PP^N$ be a morphism of projective varieties.
If $X$ satisfies $R_1$, then the pullback of a general hyperplane in $\PP^N$ satisfies $R_1$ if and only
if every $n-1$ dimensional component of $A_0(f)$ is contained in some fiber of $f$.
\end{cor}

\begin{proof} 
By Theorem \ref{main-bertini}, if either every component of $A_0(f)$ has codimension at least two,
or the divisorial components of $A_0(f)$ are contained in a fiber, then $f^{-1}(H)$ satisfies $R_1$ for a
general hyperplane $H$. To prove the 
converse it suffices to show that if $Y_0$ (defined in \eqref{yr}) has positive dimension, and $f^{-1}(H)$ satisfies $R_1$, then
$\codim(A_0(f))\,>\,1$. Notice that $A_0(f)\bigcap f^{-1}(H)$ is always a singular point of $f^{-1}(H)$, and since
$Y_0\bigcap H\,\neq\, \emptyset$ for a general $H$, regularity in codimension one of $f^{-1}(H)$ implies
that $A_0(f)\bigcap
f^{-1}(H)$ has codimension at least two in $f^{-1}(H)$. But the codimension of $A_0(f)\bigcap f^{-1}(H)$ in $f^{-1}(H)$
coincides with the codimension of $A_0(f)$ in $X$ as $H$ is a general hyperplane. Hence $A_0(f)$ is of codimension
at least two in $X$.
\end{proof}

Next we look for sufficient conditions for a complete intersection to be $R_k$. The statement is formulated in 
terms of the codimension of the differential rank loci $A_r(f)$ for various $r$. We need to understand the
differential rank of the restriction of $f$ on $f^{-1}(H)\,\longrightarrow\, H$. As before, let
$A_r(f)$ be the regular points of $X$, where $r$ is the rank of $f$.

\begin{lemma}\label{lem-b}
Let $H\, \subset\, {\PP}^N$ be a hyperplane.
Let $$f_H\,:=\,f\vert_{f^{-1}(H)}\,:\,f^{-1}(H)\,\longrightarrow\, H$$ be the restriction
map. Then $A_r(f_H)\,=\,A_{r+1}(f)\bigcap f^{-1}(H)$.
\end{lemma}

\begin{proof}
Take a point $x\,\in\, f^{-1}(H)$. From Lemma~\ref{lem-a} we
know that $x$ is a regular point of $f^{-1}(H)$ if and only if
\begin{itemize}
\item $x$ is a regular point of $X$, and

\item $df(T_xX)$ is not contained in $H$.
\end{itemize}
The differentials $df_H$ and $df$ of the maps $f_H$ and $f$
respectively satisfy the condition
$$
df_H(T_x (f^{-1}(H))) \,=\, df(T_xX)\bigcap T_{f(x)}H\,.
$$
Consequently, $y\, \in\, A_r(f_H)$ if and only if
\begin{itemize}
\item $y$ is a regular point of $f^{-1}(H)$, and

\item the rank of $df$ at $y$ is $r+1$.
\end{itemize}
Hence we conclude that $A_r(f_H)\,=\,A_{r+1}(f)\bigcap f^{-1}(H)$.
\end{proof}

\begin{thm}\label{bertini.linear.subspace}
If $X$ satisfies $R_k$ for some $1\,\leq\, k\,\leq\, n$, then the pullback of a general linear subspace of
codimension $m\,<\,n$ satisfies $R_k$ provided for each $r <\, k$ and
each $0\, \leq\, j\, <\, m$, at least one of the following two conditions hold:
\begin{enumerate}
\item The codimension of the stratum $A_{r+j}(f)\, \subset\, X$ is at least $k-r+1$.

\item The dimension of $Y_{r+j}$ in \eqref{yr} is at most $j$.
\end{enumerate}
\end{thm}

\begin{proof}
The proof will be carried out using induction on $m$. Choose a general hyperplane $H$ in ${\PP}^N$. Then the
codimension of $A_r(f_H)$ in $f^{-1}(H)$ coincides with the codimension of $A_{r+1}(f)$ in $X$,
because we may assume that $A_{r+1}(f)$ is not contained in $f^{-1}(H)$.
Note that $A_{r+j}(f_H)\,=\,A_{r+j+1}(f)\bigcap f^{-1}(H)$ for all $j\,\geq\, 0$.
Hence for each $r\,<\,k$ and $j\,<\, m-1$, we --- by assumption --- have either
$${\rm Codim}_{f^{-1}(H)}(A_{r+j}(f_H)) \,=\, {\rm Codim}_X(A_{r+j+1}(f)) \,\geq\, k-r+2 \,>\, k-r+1 $$
or
$$
\dim f_H(A_{r+j}(f_H)) \,=\, \dim f_H(A_{r+j+1}(f) \cap f^{-1}(H))
$$
$$
=\,\dim f_H(A_{r+j+1}(f)) \cap H \,\leq\, (j+1)-1
\,=\,j\, .
$$
By Theorem \ref{main-bertini}, $$X_1\,:=\, f^{-1}(H)$$ satisfies $R_k$, and since $X_1$ is dominated by the
variety obtained by pulling back $H$ to the normalization of $X$, we conclude that $X_1$ itself is a variety.
Now applying induction on $n\,=\, \dim X$ for a general hyperplane $H$, the theorem follows.
\end{proof}

\begin{cor}
Let $X$ satisfy $R_1$. Then the pullback of a general linear subspace of codimension $m$ is $R_1$ if
and only if for each $0\, \leq\, j\, <\, m$, at least one of the following two holds:
\begin{enumerate}
\item The codimension of the stratum $A_j(f)\, \subset\, X$ is at least two.

\item $\dim Y_j\, \leq\, j$.
\end{enumerate}
\end{cor}

\begin{proof} 
If $\dim(Y_j)\,>\,
j$, then a general linear subspace $H_j$ of codimension $j+1$ intersects $Y_j$ nontrivially, and every
point of $f^{-1}(H_j)$ is singular at all the points of $A_j(f)\bigcap f^{-1}(H_j)$.
Let $f_{H_j}\,=\,f\vert_{f^{-1}(H_j)}$. Then $A_0(f_{H_j})$ contains $A_j(f)\bigcap f^{-1}(H_j)$. If $A_j(f)$ is a divisor on $X$,
then $A_j(f)\bigcap
f^{-1}(H_j)$ is a divisor on $f^{-1}(H_j)$, if it is not empty. Also, since $Y_j\bigcap H_j$ is nonempty, it
follows that $A_0(f_{H_j})$
is not contained in any fiber of $f_{H_j}$. Now the result follows from Corollary \ref{criteria.R_1}
and Theorem \ref{bertini.linear.subspace}.
\end{proof}

Next we begin with the local analysis of divisorial components of $A_0(f)$.
%If the map $f\,:\,X\,\longrightarrow\, \PP^N$ has rank zero, then each
%component $f_i$ of $f\,=\,(f_1,\,\cdots ,\,f_N)$ has rank zero.
More precisely, we will study the divisorial components of
rank zero locus of a function in a regular local ring.

\begin{lemma}
Assume that the characteristic of $k$, which will be denoted by $p$, is positive.
Let $A$ be a finitely generated integral domain over $k$ of dimension $n$, and let $w\,\in\, A$ be a separable
function (i.e., the differential is not identically $0$) such that the differential of $w$ vanish along a reduced
irreducible divisor $h\,=\,0$ in ${\rm Spec}\,A$, for some $h\,\in\, A$. Then $w$, at any regular point $x$ of
${\rm Spec}\,A$ and
of $\{h\,=\,0\}$, is of the form $w\,=\, h^2g_1 + g_2^p$ for some $g_1,\,g_2\,\in\, {\mathcal M}_x$, where
${\mathcal M}_x$ is the maximal ideal of the completion ${\mathcal A}_x$ of $A$ along the maximal ideal of $x$. 
\end{lemma}

\begin{proof}
The locus where $w$ has rank zero is independent of the
choice of coordinates. Assume that this locus has a divisorial 
component with generic point contained in the regular locus of ${\rm Spec}~(A)$. 
Consider this divisor $D$ with reduced structure together
with a general point $x$ where this divisor is smooth. We choose coordinates $(x_1,\,x_2,\,\cdots ,\,x_n)$ near the point $x$
 such that $h\,=\,x_1\,=\,0$ is the divisor $D$ along which $w$ has rank $0$. Now consider the
decomposition of $$w\,=\,g+(g')^p\,\in\, {\mathcal A}_x=k[[x_1,\ldots,x_n]]$$ with $g,\,g'\,\in \,{\mathcal M}_x$ and none of the monomials
in $g$ is a $p$--th power. Let 
$$\Sigma_w\,:= \,\langle\partial_{x_i}(w)\rangle\,=\,\langle\partial_{x_i}(g)\rangle \,=:\, \Sigma_{g}$$
be the ideal of critical space generated by the partial 
derivatives of $w$. Since we assumed that $w$ is separable, this ideal is not the zero ideal.

Since $w$ is singular
along $x_1\,=\,0$, it follows that $\sqrt{\Sigma_w}\,\subset\, (x_1)$, where $\sqrt{~~}$
denotes the radical of the ideal. Let $g\,=\,a+x_1b+x_1^2c$, where both
$a$ and $b$ are elements of ${\mathcal M}_x$ that are not a $p$--th power and are expressible as power
series in the variables $x_2,\,\cdots,\, x_n$ only. Since $x_1$ divides $\partial_{x_i}(g)$ for all $i$, it follows that 
$x_1$ divides $\partial_{x_i}(a)$, which implies that $\partial_{x_i}(a)\,=\,0$ for all $i$ because $a$ is a power
series in the variables $x_2, \,\cdots,\, x_n$ only. But a function whose all partial
derivatives are identically zero is either 
a constant or a $p$--th power. By our assumption that
the element $a$ is in the maximal ideal and not a $p$--th power, we
conclude that $a\,=\,0$. Hence $g\,=\, x_1b + x_1^2c$.

Continuing, $x_1$ divides
$\partial_{x_1}(g) \,=\, b+x_1\partial_{x_1}(b)+2x_1c+x_1^2\partial_{x_1}(c)$
if and only if $x_1$ divides $b$, which implies that $b\,=\,0$ as $b$ is a power series in other variables. 
Hence $w\,=\,x_1^2g_1 + g_2^p$ by taking $g_1\,=\,c$ and $g_2\,=\,g'$, proving the lemma.
\end{proof}

\begin{example}\label{ex210}
Let $\mathbf k$ be an algebraically closed field of characteristic $p\,>\,0$, and let $K$ be the field extension of
$k(x,y)\,=\,k(\PP^2)$ given by attaching roots of $Z^p-xZ -y$. Let $f\,:\, X\,\longrightarrow\,
\PP^2$ be the normalization map. Then $f$ is not residually separable, yet $f^{-1}(H)$ is a connected nonsingular
curve for $H$ different from the three line $x\,=\,0$, $y\,=\,0$ and the line at infinity. Note that in this
case $A_1(f)\,=\,f^{-1}(\{x=0\})$ and
$A_2(f)\,=\,X\setminus A_1(f)$. Hence Theorem \ref{main-bertini} applies (with $k=1$) though Corollary \ref{residually-separable} doesn't.
\end{example}

\begin{example}\label{counter-example-normal}
Let $\mathbf k$ be as in Example \ref{ex210} and $p$ an odd prime.
Let $f\,:\,\PP^2\,\longrightarrow\, \PP^2$ be a morphism given by
$$f(x,y,z)\,=\,(x^p-yz^{p-1},\, x^p+yz^{p-1},\, xz^{p-1}+ y^p+z^p)\, .$$
The Jacobian matrix for $f$ is
\[
\begin{bmatrix}
    0       & -z^{p-1} & yz^{p-2} \\
    0       &  z^{p-1} & -yz^{p-2}\\
    z^{p-1} & 0        & -xz^{p-2} \\
\end{bmatrix}
\]
Note that $A_2(f)\,=\,(\{z\ne 0\})$
and $A_0(f)\,=\,(\{z=0\})$, hence the hypothesis of Theorem \ref{main-bertini} fails. Note that
the pullback of any line in $\PP^2$ by $f$ is singular at $z\,=\,0$.
\end{example}

A variety $X$ is said to satisfy \emph{$S_k$} if for every point $x\,\in\, X$ of codimension $r$, the stalk
${\mathcal O}_{X,x}$ has 
a regular sequence of length at least $\min(k,\,r)$, i.e., X is Cohen-Macaulay at points of codimension at most $k$.

\begin{lemma} \label{hyperplane-S_k}
Let $X$ be a variety satisfying $S_k$. Then the subset $Z\, \subset\, X$ where $S_{k+1}$ fails is of codimension
at least $k+1$. 
\end{lemma}

\begin{proof}
Suppose $Z$ has codimension $k$. By taking an irreducible component of it we may assume that $Z$ is irreducible. In view
of the local nature of the statement it suffices to prove it under the assumption that
$X$ is affine. Let $P$ be the prime ideal in the coordinate ring of $X$
defining $Z$. Then $S_k$ implies that $P$ contains a regular sequence of length $k$, say $a_1,\, a_2,\, \cdots,
\, a_k$. Let $I$ be the ideal $(a_1,\,a_2,\,\cdots,\, a_k)$. Since $X$ is of finite type over a field, there are
infinitely many codimension $k+1$ primes containing $P$, and there are finitely many associated primes of $I$. So
there exists a codimension $k+1$ prime $Q$ containing $P$ such that $Q$ is not contained in the union of the
primes associated to $P$. Take $a_{k+1}$ in $Q$ such that $a_{k+1}$ is not contained in any of the
primes associated to $I$. Then $a_1,\,a_2,\,\cdots,\,a_{k+1}$ is a regular sequence in $Q$ contradicting the
assumption that $S_{k+1}$ fails for $Q$. This completes the proof.
\end{proof}

\begin{cor}\label{S_kforhyplanesection}
Let $X$ be a $n$-dimensional projective variety satisfying $S_k$. Let $Z$ be the closed subset of $X$ where
$S_{k+1}$ fails, and let $H$ be a hypersurface not containing any component of $Z$ of dimension $n-k-1$
(maximum possible dimension). Then $H$ satisfies $S_k$.
\end{cor}

\begin{proof} 
Let $X$ satisfy $S_k$. Then by definition, for any $H$, the intersection $X\cap H$ satisfies $S_{k-1}$. Let $H$ be a
general hyperplane such that $H\cap Z$ is a subscheme of codimension at least $k+1$ in $H$. Now every point $y$
of codimension $k$ in $X\cap H$ is also a point of codimension $k+1$ in $X$. Hence $y$ is a $n-k-1$ dimensional point.  Note that $y$ cannot be a component
of $Z$ as $H$ intersects $n-k-1$ dimensional component of $Z$ properly while $y$ is a point in $X\cap H$. From Lemma~\ref{hyperplane-S_k} we know
that $X$ satisfies $S_{k+1}$ along $y$. Hence $X\cap H$ satisfies $S_k$ along $y$. 
\end{proof}

\begin{lemma}\label{bertini-normal-finite-mor}
Let $Z_0$ be a normal projective variety over $\mathbf k$. Let 
$$f\,:\,Z\,\longrightarrow\, Z_0$$ be a finite separable morphisms of normal varieties such that $A_j(f)$ has 
codimension at least two in $Z$ for all $0\,\leq\, j\,\leq \,\dim(Z_0)-2$. For a hyperplane section $W_0$ in $Z_0$, let 
$$W\,:=\,f^{-1}(W_0)$$ be the inverse image in $Z$. For a general $W_0$, the hypersurface $W$ is an irreducible normal variety, and $A_j(f\vert_{W})$ has 
codimension at least two in $W$ for all $0\,\leq\, j\,\leq\, \dim(Z_0)-3$.

Moreover, the above holds
when $W_0$ is replaced by a general linear subspace of codimension at most $\dim(Z_0)-1$. 
\end{lemma}

\begin{proof}
This follows from Corollary \ref{criteria.R_1}, Theorem \ref{bertini.linear.subspace} and Lemma \ref{hyperplane-S_k}.
Note that since $A_r(f_i)$ has codimension at least two for all $r\,<\,\dim(Z_0)-1$, using induction, the two
statements in the lemma holds with $W_0$ replaced by a general linear subspace of codimension at most $\dim(Z_0)-1$.
\end{proof}

\begin{lemma}\label{bertini-divisor}
Continue with the notation and the hypothesis of the Lemma \ref{bertini-normal-finite-mor}. Let $D$ be a 
divisor in $Z$ such that $f$ restricted to each irreducible component of $D$ is separable. Then for a 
general $W_0$, the inverse image $W\,=\, f^{-1}(W_0)$ intersects $D$ transversally on a nonempty Zariski 
open dense subset of $W\cap D$.

Moreover, if $W_0$ is replaced by a general linear subspace of codimension $\dim(Z_0)-1$ (this implies
that $W_0$ is a curve), then $W$ is a smooth curve in $Z$ which intersects $D$ transversally. 
\end{lemma}

\begin{proof}
Note that since $f$ restricted to any irreducible component of $D$ is separable, for a general hyperplane 
section $W_0$ the intersection
$D\cap W$ will be a divisor in $W$, and $f$ restricted to $D\cap W$ will be separable.
Hence $W$ intersects $D$ transversally away from a proper closed subset of every irreducible component of 
$W\cap D$.

Therefore, in view of Lemma \ref{bertini-normal-finite-mor}, we can apply induction to conclude that if $W_0$ is a 
section obtained by a general linear subspace of codimension $\dim(Z_0)-1$, then $W$ is a smooth curve in $Z$,
and since $W\cap D$ is a finite set of points, $D$ intersects $W$ transversally.
\end{proof}

\section{Lefschetz theorem for normal varieties}\label{sec:lefschetz-normal}

Given a projective variety $Z$ defined over $\mathbf k$, let ${\rm Vect}^{f}(Z)$ denote the category of essentially
finite vector bundles on $Z$ \cite{No}; it is an abelian
tensor category. Given a base point $x_0\, \in\, Z$, we have a fiber functor on
${\rm Vect}^{f}(Z)$ that sends any $E\, \in\,{\rm Vect}^{f}(Z)$ to its fiber $E_{x_0}$. The Nori fundamental
group $\pi^N(Z,\, x_0)$ is defined to be the corresponding Tannaka dual \cite{No}. There is a natural surjection
from $\pi^N(Z,\, x_0)$ to the \'etale fundamental group $\pi_1^{et}(Z,\, x_0)$.

We note that the injectivity part of Lefschetz theorem for \'etale fundamental group
fails in general for normal varieties. To give such an
example, let $X\,\subset\,\PP^n$ be a projectively normally embedded smooth variety with 
nontrivial $\pi^{et}(X)$. Let $\widehat{X}\,\subset\,\PP^{n+1}$ be its cone; note that $\widehat X$ is normal.
Then $\widehat{X}\cap H$ is isomorphic to $X$ for a 
general hyperplane $H$. We will show that $\widehat{X}$ is simply connected. Let 
$\widetilde{X}\,\longrightarrow\,\widehat X$ be the blow-up of the vertex. Then the projection
$q\, :\, \widetilde{X}\,\longrightarrow\,X$ is a $\PP^1$--bundle over $X$. Now by the homotopy 
exact sequence for \'etale fundamental groups \cite[p.~117]{Mu}, the map
$q$ induces an isomorphism of the \'etale fundamental groups. 
The exceptional locus $E$ of $\widetilde{X}\,\longrightarrow\, \widehat X$ defines a section of $q$. Hence the
inclusion map $E\,\hookrightarrow\,\widetilde{X}$ 
induces an isomorphism of \'etale fundamental groups. Let $Y$ be a connected \'etale cover of $\widehat X$, 
then $Y$ is normal and hence irreducible.
Let $\widetilde{Y}$ be its base change to $\widetilde{X}$, which is birational to $Y$ and hence is also irreducible. 
This \'etale covering is evidently trivial over $E$, and hence it is trivial over $\widetilde X$ and an isomorphism as 
$\widetilde{Y}$ is irreducible. This implies that $Y\,\longrightarrow\,\widehat{X}$ is also an isomorphism.
Hence $\widehat{X}$ is simply connected.

However, the surjectivity part of Lefschetz theorem for \'etale fundamental group holds for normal varieties.

The above example 
also implies that the Nori fundamental group is not an isomorphism under restriction to general hyperplane 
sections of $X$. On the other hand, as we shall prove in this
section, the surjectivity of the Nori fundamental groups, which is 
known for smooth projective varieties \cite{BH}, continues to hold for normal projective varieties.

The following result is a slight variation of \cite[Chapter III, Theorem 7.6]{har} which uses Serre duality
and Serre vanishing.

\begin{pro}\label{serre-vanishing}
 Let $X$ be an equidimensional projective scheme over an algebraically closed field such that for all closed points $x$ in $X$, depth of $\cO_{X,x}$ is $n$. Let $\scrV$ be a locally free sheaf on $X$. Then $H^i(X,\scrV(-d))\,=\,0$ for $i\,<\,n$ for $d$ sufficiently large.
\end{pro}

\begin{proof}
The proof is exactly same as in \cite[Chapter III, Theorem 7.6(b)]{har} of $(i) \implies (ii)$; note that the depth
of $\scrV_x$ is same as that of $\cO_{X,x}$.
\end{proof}

\begin{lemma}\label{Q}
Let $X$ be a normal projective variety. Let $F\,:\,X\,\longrightarrow\, X$ be the Frobenius morphism,
and let $Q$ be the cokernel of the induced map $\cO_X\,\longrightarrow\, F_*\cO_X$. Let $\scrV$ be a vector
bundle on $X$. Then $H^0(X,\,\scrV \otimes Q(-d))\,=\,0$ for some $d$ sufficiently large. In particular,
associated primes of $Q$ have positive dimension.
\end{lemma}

\begin{proof}
Note that if $X$ is a curve then it is smooth. Hence the result follows from \cite[Chapter III,
Theorem 7.6(b)]{har} and the fact that $Q$ is locally free (see the remark below). So assume that
$\dim X \, \geq\, 2$.

Tensoring the exact sequence
$$0\,\longrightarrow\, \cO_X\,\longrightarrow\, F_*\cO_X\,\longrightarrow\, Q\,\longrightarrow\, 0$$
 by the vector bundle $\scrV(-d)$, we obtain the exact sequence
$$0\,\longrightarrow\, \scrV(-d)\,\longrightarrow\, F_*\cO_X\otimes \scrV(-d)\,\longrightarrow\,
\scrV\otimes Q(-d)\,\longrightarrow\, 0.$$
Consider the associated long exact sequence of cohomologies
 $$\cdots \,\longrightarrow\, H^0(X,\, F_*\cO_X\otimes\scrV(-d))\,\longrightarrow\,
H^0(X,\,\scrV\otimes Q(-d)) \,\longrightarrow\, H^1(X,\,\scrV(-d))\,\longrightarrow\, \cdots\, .$$
We have $H^0(X,\, F_*\cO_X\otimes\scrV(-d))\,=\,0$ for $d$ sufficiently large as $F_*\cO_X\otimes\scrV$ is
torsionfree. To see that $H^1(X,\,\scrV(-d))\,=\,0$, first note that $X$ is normal. Since
$\dim X\, \geq\, 2$, the depth of $\cO_{X,x}$ is at least 2 for $x$ any closed point of $X$. Now the
result follows from Proposition \ref{serre-vanishing}.

For the last statement, suppose $x$ be a closed point in $X$ and the corresponding maximal ideal is an
associated prime of $Q$. Then there exist a nonzero section $s$ of $Q$ such that the support of $s$ is
precisely $x$. Hence we get a morphism $i_*\cO_{X,x}/m_x\,\longrightarrow\, Q$ sending $1$ to $s$
where $i\,:\,x\,\longrightarrow\, X$ is the closed immersion. Tensoring with $\cO(-d)$ for
$d$ sufficiently large and taking global section produces a section of $Q(-d)$ contradicting
the fact that $H^0(X,\, Q(-d))\,=\,0$.
\end{proof}

\begin{remark}
We note that if $X$ is smooth then $Q$ is isomorphic to $F_*B^1_X$ where $B^1_X$ is the subsheaf of
$\Omega^1_X$ consisting of exact differentials. %Lemma \ref{Q} implies that the same is true if
%$X$ is normal and it has only isolated singularities. 
\end{remark}

For a sheaf of $\cO_X$-module $\scrF$ on $X$, let $\tor(\scrF)$ denote the torsion subsheaf of of $\scrF$.

\begin{pro}\label{proQ}
Take $Q$ as in Lemma \ref{Q}. For any associated prime $P$ of the sheaf $Q$, let $i_P
\,:\,X_P\,\longrightarrow\, X$ be the closed immersion, where $X_P$ is the closed subvariety of
$X$ defined by $P$. Let $Q_P\,=\,i_P ^*Q/\tor(i_P ^*Q)$ be torsionfree on $X_P$. Let $\scrV$ be
an essentially finite bundle on $X$ and $Y$ be a hypersurface on $X$. If $\degree(Y)\, >\,
\mu_{max}(Q_P)$ for every associated prime $P$ of $Q$ then $H^0(X, Q\otimes \scrV(-Y))\,=\,0$.
\end{pro}

\begin{proof}
 By Lemma \ref{Q} for every associated prime $P$ of $Q$, the dimension of $X_P$ is at least one. Let $s
\,\in\, H^0(X,\, Q\otimes \scrV(-Y))$ be a nonzero section. Then there exist $P$ associated prime of $Q$
such that $X_P \,\subset\, \Supp(s)$. The image $s_P$ of $s$ under the map $Q\otimes \scrV(-Y)\,\longrightarrow\,
i_{P*}(Q_P\otimes i^*\scrV(-Y))$ is also nonzero. But this implies that
$H^0(X_P, \,Q_P\otimes i_P^*\scrV(-Y))\,=\,\Hom_{X_P} (i_P^*\scrV^{\vee}(Y),\, Q_P)$ is nonzero.

As $\scrV$ is essentially finite, the pullback $i_P^*\scrV^{\vee}$ is essentially finite on $X_P$, and hence
$i_P^*\scrV^{\vee}(Y)$ is semistable of $\degree(Y)$. Now the inequality $$\degree(Y)\, > \, \mu_{max}(Q_P)$$
implies that $\Hom_{X_P} (i_P^*\scrV^{\vee}(Y),\, Q_P)\,=\,0$. This contradicts the earlier observation that
$\Hom_{X_P} (i_P^*\scrV^{\vee}(Y),\, Q_P)\,\not=\,0$.
\end{proof}

\begin{thm}\label{thmnl}
Let $X$ be a normal projective variety, of dimension at least two, defined
over an algebraically closed field $\mathbf k$. Let $Q$ and $Q_P$ be as in Proposition \ref{proQ} and
$$d\,\,=\,\,{\rm max}\{\mu_{max}(Q_P)\,\big\vert\,\, P \text{ associated prime of } Q\}.$$ Then for a general
ample hypersurface $Y\,\subset\, X$ of degree greater than $d$, the induced homomorphism
$\pi^N(Y)\,\longrightarrow\, \pi^N(X)$ is surjective (faithfully flat).
\end{thm}

\begin{proof}
Let
\begin{equation}\label{ei}
i\, :\, Y\, \hookrightarrow\, X
\end{equation}
be the inclusion map.
Consider the functor $i^*\, :\, {\rm Vect}^{f}(X)\,\longrightarrow\, {\rm Vect}^f(Y)$, between
the categories of essentially finite vector bundles, that sends any vector bundle $W$ on $X$ to $i^*W$. To
prove the theorem it suffices to show the following:
\begin{enumerate}
\item $i^*$ is fully faithful, and

\item for any $E\, \in\, {\rm Vect}^{f}(X)$, every essentially finite subbundle of $i^* E$ is of the
form $i^*{\mathcal F}$ for some essentially finite subbundle
${\mathcal F}\, \subset\, E$.
\end{enumerate}
(See \cite[p.~139, Proposition 2.21(a)]{DM}.)

For essentially finite vector bundles $V_1$ and $V_2$ on $X$, define
$${\mathcal E} \,:=\, {\mathcal H}om (V_1,\,V_2)\,=\, V_2\otimes V^*_1\, ,$$
which is also essentially finite \cite[p.~82, Proposition 3.7(c)]{No}. To prove the above
statement (1), we need to show that the restriction homomorphism
\begin{equation}\label{is}
i^*\, :\, H^0(X,\, {\mathcal E})\, \longrightarrow\, H^0(Y,\, i^*{\mathcal E})
\end{equation}
is an isomorphism.

For notational convenience, we shall denote the vector
bundle ${\mathcal E}\otimes {\mathcal O}_X(-Y)$ by ${\mathcal E}(-Y)$.

Tensor the short exact sequence of sheaves on $X$
$$0\,\longrightarrow\, {\mathcal O}_X (-Y)\,\longrightarrow\, {\mathcal O}_X
\,\longrightarrow\, i_*{\mathcal O}_Y \,\longrightarrow\, 0$$
with ${\mathcal E}$. Let
$$
0 \, \longrightarrow\,
H^0(X,\, {\mathcal E}\otimes {\mathcal O}_X(-Y))\, \longrightarrow\,
H^0(X,\, {\mathcal E})\, \stackrel{i^*}{\longrightarrow}\, H^0(Y,\, i^*{\mathcal E})
$$
$$
\longrightarrow\,H^1(X,\, {\mathcal E}(-Y)) \, \longrightarrow\, \cdots
$$
be the corresponding long exact sequence of cohomologies. From this exact sequence we conclude that $i^*$ in
\eqref{is} is an isomorphism if
\begin{equation}\label{is2}
H^0(X,\, {\mathcal E}(-Y))\,=\,0\, =\, H^1(X,\, {\mathcal E}(-Y))\, .
\end{equation}
Now, $H^0(X,\, {\mathcal E}(-Y))\,=\,0$ because ${\mathcal E}(-Y)$
is a semistable vector bundle of negative degree.

Consider the Frobenius map $F\, :\,X\,\longrightarrow\, X$. We have the short exact sequence
\begin{equation}\label{k3}
0\, \longrightarrow\, {\mathcal O}_X \,\longrightarrow\, F_*{\mathcal O}_X \,\longrightarrow\, Q
\, \longrightarrow\, 0
\end{equation}
of coherent sheaves on $X$ where the first map is locally given by $x\longmapsto x^p$ and $Q$ is the cokernel of this map. Tensoring it with ${\mathcal E}(-Y)$, we get the short exact sequence
\begin{equation}\label{is3}
0\, \longrightarrow\, {\mathcal E}(-Y)\, \longrightarrow\, (F_*{\mathcal O}_X)\otimes{\mathcal E}(-Y)
\, \longrightarrow\,Q\otimes{\mathcal E}(-Y) \, \longrightarrow\, 0
\end{equation}
on $X$.

Since the sequence in \eqref{k3} is an exact sequence of coherent sheaves,
tensor product of \eqref{k3} with ${\mathcal E}(-Y)$ remains exact as ${\mathcal E}(-Y)$ is locally free. The
long exact sequence of cohomologies associated to \eqref{is3} gives
$$
\cdots \, \longrightarrow\, H^0(X,\, Q \otimes {\mathcal E}(-Y))\, \longrightarrow\, H^1(X,\, {\mathcal E}(-Y))
\,\longrightarrow\,H^1(X,\, F_*{\mathcal O}_X\otimes{\mathcal E}(-Y))\,\longrightarrow\, \cdots \, .$$

Hence to prove that $H^1(X,\, {\mathcal E}(-Y))\,=\, 0$ (see \eqref{is2})
it suffices to show that
\begin{equation}\label{k4}
H^0(X, \,Q\otimes{\mathcal E}(-Y))\, =\,0\,=\,
H^1(X,\, F_*{\mathcal O}_X\otimes{\mathcal E}(-Y))\, .
\end{equation}

Now $ {\rm deg}~Y\,>\, \mu_{\rm max}(Q_P)$ for every associated prime $P$ of $Q$. Hence by Proposition \ref{proQ},
$H^0(X,\, Q \otimes{\mathcal E}(-Y)) \,=\,0$.

So to prove \eqref{k4} it remains to show that $H^1(X,\, F_*{\mathcal O}_X\otimes{\mathcal E}(-Y))\, =\,0$.

To prove that $H^1(X,\, F_*{\mathcal O}_X\otimes {\mathcal E}(-Y))\, =\,0$, note that 
\begin{align*}
H^1(X,\, F_*{\mathcal O}_X\otimes{\mathcal E}(-Y))&\,=\,H^1(X,\,F_*F^*{\mathcal E}(-Y)) \text{ (projection formula)}\\
&=\,H^1(X,\,F^*({\mathcal E}(-Y))) \text{ (finiteness of Frobenius map)}\\
&=\,H^1(X,\,F^*{\mathcal E}\otimes\cO_X(-pY)) \\
&\subset\, H^1(X,\,F_*F^*{\mathcal E}\otimes \cO_X(-pY))\\
& \text{ (as }\, H^0(X,\, Q \otimes F^*{\mathcal E}(-Y))\,=\,0 \text{ (Proposition \ref{proQ}))}\\
&=\,H^1(X,\,F^{2*}{\mathcal E}\otimes\cO_X(-p^2Y)) \text{~ (applying previous steps)}\\
&\subset\, H^1(X,\,F^{3*}{\mathcal E}\otimes\cO_X(-p^3Y)) \,\subset\, \cdots .
\end{align*}
We have $\text{rank}(F^{r*}{\mathcal E})\,=\, \text{rank}({\mathcal E})$,
and $F^{r*}{\mathcal E}$ is strongly semistable of slope zero; hence $H^1(X,
\,F^{r*}{\mathcal E}\otimes\cO_X(-p^rY))\,=\,0$ for $r$ sufficiently large (depending on the rank of ${\mathcal E}$)
by boundedness of semistable sheaves of fixed
numerical invariant \cite{La1}. Consequently, we have $H^1(X,\, F_*{\mathcal O}_X\otimes
{\mathcal E}(-Y))\, =\,0$.

This proves \eqref{k4}. As noted before, this implies that $i^*\, :\, {\rm Vect}^{f}(X)\,\longrightarrow\, {\rm 
Vect}^f(Y)$ is fully faithful.

Now we will show that every essentially finite subbundle of $i^*E$ is of the form $i^*F$ for some
subbundle $F$ of $E$.

For a semistable bundle $V$, its socle, which is denoted by $Soc(V)$, is the unique maximal polystable subsheaf
of $V$ \cite[p.~23, Lemma 1.5.5]{HL}, \cite[p.~1034, Proposition 3.1]{BDL}.

It can be shown that the socle $Soc(V)$ of an
essentially finite vector bundle $V$ on $X$ is an essentially finite subbundle of $V$. To prove this, we recall
that the condition that $V$ is essentially finite implies that there is a connected \'etale Galois covering
$$
\varpi\, :\, Y\, \longrightarrow\, X\, ,
$$
and a positive integer $m$, such that $(F^m_Y)^*\varpi^*V$ is trivial, where $F_Y\, :\, Y\, \longrightarrow\, Y$
is the Frobenius morphism of $Y$ (see \cite[p.~557]{BH}). If $W\,\subset\,{\mathcal O}^{\oplus r}_Y$ is
a coherent subsheaf such that $\text{degree}(W)\,=\, 0$ and ${\mathcal O}^{\oplus r}_Y/W$ is torsionfree, then
$W$ is actually a subbundle of ${\mathcal O}^{\oplus r}_Y$ and it is in fact trivializable. In particular,
$(F^m_Y)^*\varpi^*Soc(V)$ is trivializable. This implies that $Soc(V)$ is an
essentially finite vector bundle (see \cite[Proposition 2.3]{BH}).

First note that for any subbundle $W\,\subset\,E$ of degree zero we have
$Soc(W) \,=\, Soc(E) \cap W$. It is evident from the definition that $Soc(E)\cap W
\,\subset\, Soc(W)$. If $W_1$ is an stable subbundle of $W$ of degree zero, 
then $W_1\,\subset\, Soc(E)$ as $Soc(E)$ is generated by all stable subbundles
of degree zero. This implies that that reverse inclusion $Soc(E)\cap W\,\supset\, Soc(W)$ holds. 

Now we show that $Soc(i^*E)$ coincides with $i^*Soc(E)$, where $i$ is the map in \eqref{ei}. Clearly
$$Soc(i^*E)\,\supset\, i^*Soc(E)$$ by a variant of Bogomolov restriction theorem, (\cite[Theorem 0.5]{La2}, also see \cite{Bo}, \cite{La1}), that stable bundles restrict to stable bundles
when the degree of the ample hypersurface is sufficiently large. We note that the this result of \cite{La2}
is under the assumption that the characteristic of the base field is positive.
In the case characteristic zero, the result follows from the result on positive characteristics by the usual
spreading out argument and the openness of the stability condition (in mixed characteristic)
\cite[p.~635, Theorem 2.8(B)]{Ma}. (We thank the referee for pointing this out.)
As $Soc(E)$ is essentially finite, so its chern classes $c_1$ and $c_2$ are zero. Hence there is no restriction on the  degree of the ample hypersurface. 
Let $W$ be a stable 
bundle on $Y$. If $W$ is a subbundle of $i^*E/i^*Soc(E)$, then by induction on the rank there exists a subbundle 
$\widetilde W$ of $E/Soc(E)$ on $X$ such that $i^*\widetilde{W}\,=\,W$. Consider the following commutative 
diagram:
\begin{equation}\label{t1}
\begin{matrix}
H^0(X,\, Hom(\widetilde{W},\,E)) & \stackrel{\alpha}{\longrightarrow} & H^0(X,\, Hom(\widetilde{W},\,E/Soc(E)))\\
\Big\downarrow && \Big\downarrow\\
H^0(Y,\, Hom(W,\,i^*E))&\stackrel{i^*\alpha}{\longrightarrow} & H^0(Y,\, Hom(W,\,i^*E/i^*Soc(E)))
\end{matrix}
\end{equation}
Recall that for any semistable bundle $V$ of degree zero, a subbundle $S$ of
degree zero contains $Soc(V)$ if and 
only if for any stable bundle $W$ of degree zero the natural homomorphism $H^0(Y,\, Hom(W,\,V))\,\longrightarrow\, 
H^0(Y,\, Hom(W,\,V/S))$ is the zero map.

Now by the definition of a socle, $\alpha$ is the zero map. By fully faithfulness of $i^*$, vertical arrows in 
\eqref{t1} are isomorphisms. Hence $i^*\alpha$ is also the zero map. If $W$ is not a subbundle of 
$i^*E/i^*Soc(E)$ then $H^0(Y,\, Hom(W,\,i^*(E/Soc(E)))\,=\,0$ and hence again $H^0(Y,\, 
Hom(W,\,i^*E))\,\longrightarrow\, H^0(Y,\, Hom(W,\, i^*E/i^*Soc(E)))$ is the zero map. Consequently, we have 
$i^*Soc(E)\,\supset\, Soc(i^*E)$.

Now let $W\,\subset\, i^*E$ be an essentially finite subbundle. Since Socle filtration of $W$ is induced from the 
socle filtration of $E$, we conclude that $Soc(W)$ lifts to a subbundle $W'$ of $Soc(E)$. By induction on the rank, $W/Soc(W)$ lifts 
to a subbundle $W''$ of $E/Soc(E)$. We have $H^0(X,\, Hom(W'',\, E/W'))\,=\,H^0(Y,\, Hom(W/Soc(W),\, i^*E/Soc(W)))$,
hence $W''$ is subbundle of $E/W'$. 
Let $\widetilde{W}$ be the inverse image of $W''$ in $E$. Then $i^*\widetilde{W} \,= \,W$. Therefore,
every essentially finite subbundle of $i^*E$ is of the form $i^*F$ for some
subbundle $F$ of $E$.
\end{proof}

\section{Lefschetz theorem for $\pi^N(X,P)$}

In this section we shall prove a variant of Lefschetz theorem for Nori fundamental group of formal orbifolds 
under a certain hypothesis.

We refer the reader to \cite[Section 2]{Nori-pi1} (and \cite{formal.orbifolds} for further details) for the 
definitions of branch data $P$ on a normal variety $X$, a formal orbifold $(X,\,P)$, their \'etale covers and 
their \'etale fundamental group. Also recall that for a finite cover $f\,:\,Z\,\longrightarrow \, X$ between
normal varieties, $B_f$ 
is the branch data on $X$ induced from $f$. A formal orbifold $(X,\,P)$ is called geometric if there exists an 
\'etale cover $f\,:\,(Z,\,O)\,\longrightarrow \, (X,\,P)$, where $O$ is the trivial branch data. We recall
from \cite[Section 2]{Nori-pi1} the 
definition of the Nori fundamental group $\pi^N(X,\,P)$ of a geometric formal orbifold $(X,\,P)$ from 
as the Tannaka dual of the Tannakian category $\Vect^f(X,\,P)$ of essentially 
finite vector bundles on $(X,\,P)$. For $X$ a curve these notions were 
introduced in \cite{KP}.

Assume that $(X,\,P)$ is a smooth projective geometric formal orbifold over an algebraically closed field
$\mathbf k$. Let $Y$ be 
a normal connected hypersurface of $X$ not contained in $\BL(P)$, and let $\scrp$ be the ideal sheaf defining $Y$ in 
$X$. We recall from \cite{formal.orbifolds} the definition of a branch data $P_{|Y}$. For a point $x\,\in\, Y$ 
of co-dimension at least one, let $U$ be an open affine neighborhood of $x$ in $X$. Note that 
$\scrp\widehat{\cO_X(U)}^x$ is a prime ideal in $\widehat{\cO_X(U)}^x$. Let $R$ be the integral closure of 
$\widehat{\cO_X(U)}^x$ in $P(x,U)$, and let $\scrq_1,\,\cdots,\,\scrq_r$ be the prime ideals of $R$ lying above $\scrp$. 
Define $P_{|Y}(x,U\cap Y)$ to be the compositum of Galois extensions $\QF(R/\scrq_i)$ of 
$\widehat{\cO_X(U)}^x/\scrp$. We shall denote by
$P^g_{|Y}$ the maximum geometric branch data on $Y$ which is less than or equal to $P_{|Y}$.

For an \'etale covering $(Z,\,O)\,\longrightarrow\, (X,\,P)$, let $W$ denote the normalization of the fiber 
product $Z\times_X Y$, and let $g\,:\,W\,\longrightarrow\, Y$ be the projection map. Then $g\,:\, 
(W,\,g^*P^g_{|Y})\,\longrightarrow\, (Y,\,P^g_{|Y})$ is an \'etale cover by \cite[Proposition 
8.3]{formal.orbifolds}.

\begin{pro}
Let $f\,:\,(Z,\,O)\,\longrightarrow\,(X,\,P)$ be an \'etale morphism of formal orbifolds. Let $Y\,\subset\, X$
be a normal hypersurface such that $W\,=\,f^{-1}(Y)$ is an irreducible normal hypersurface, and let $g$ denote
the restriction $f_{|W}$. Then $g$ is a finite morphism between normal varieties, and $P_{|Y}\,=\,B_g$.
\end{pro}

\begin{proof}
Since $W\,=\,f^{-1}(Y)\,=\,Z\times_X Y$ is irreducible and normal, it follows that $g\,:\,(W,\,O)
\,\longrightarrow\, (Y,\,P_{|Y})$ is an \'etale cover. Hence $P_{|Y}\,=\,B_g$.
\end{proof}

\begin{definition}\label{super-separable}
A geometric formal orbifold $(X,\,P)$ is said to be \emph{1-separable} if there exists an \'etale Galois
cover $f\,:\,(Z,\,O)\,\longrightarrow\, (X,\,P)$ such that the codimension of $A_0(f)$ in $Z$ is at least
two. More generally, it is $m$-separable if $\codim(A_j)\,>\,1$ for $j\,\le\, m-1$ and $\codim(A_m)
\,>\,0$. Moreover if the above holds for $m\,=\,\dim(X)-1$, and the restriction of $f$ to all
the irreducible components of the ramification divisor of $f$ is separable, then
$(X,\,P)$ is called \emph{super-separable}.
\end{definition}

\begin{remark}
If $(X,\,P)$ is 1-separable then for any Galois \'etale cover $f\,:\,(Z,\,O)\,\longrightarrow\, (X,\,P)$ the
codimension of $A_0(f)$ in $Z$ is at least two.
\end{remark}

\begin{thm}\label{lefschetz-thm}
Let $(X,\,P)$ be a projective geometric formal orbifold, over an algebraically closed field $\mathbf k$, which is
1-separable. Then there exists a normal hypersurface $Y\,\subset\, X$ such that $\pi^N(Y,\,P_{|Y})
\,\longrightarrow\, \pi^N(X,\,P)$ is surjective with $Y_{sing}\,\subset\, X_{sing}$.
 
Moreover, if $(X,\,P)$ is smooth, and there is a smooth \'etale cover $(Z,\,0)\,\longrightarrow\, (X,\,P)$, then
the above map is an isomorphism if $\dim X\,\geq\, 3$.
\end{thm}

\begin{proof}
Since $(X,\,P)$ is 1-separable geometric, we know that there is a normal \'etale $G$-Galois cover $f\,:\, 
(Z,\,0)\,\longrightarrow\, (X,\,P)$ such that $\codim(A_0(f))\,>\,1$. Consider a very ample line bundle 
${\mathcal L}$ on $X$. Then sections that are pull back of
sections of ${\mathcal L^m}$ form a sublinear system $f^* 
|{\mathcal L^m}|$ of the complete linear system given by the sections of $f^*{\mathcal L^m}$ for all 
$m\,\geq\, 1$. Since ${\mathcal L}$ is very ample, it is in particular base point free and hence so is $f^* 
|{\mathcal L^m}|$. Then by Bertini's theorem, a general section $W\,\subset\, Z$ is normal. This induces a 
surjection $\pi^N(W)\,\longrightarrow\, \pi^N (Z)$ of Nori fundamental group schemes when $\dim Z \,\geq\, 2$ 
for $m \,>>\,0$ (see Theorem \ref{thmnl}).

Furthermore, if $Z$ is smooth and $\dim Z\,\geq\, 3$, then the homomorphism $\pi^N(W)
\,\longrightarrow\, \pi^N (Z)$ is injective as well
\cite[p.~549, Theorem 1.1]{BH}.

By Bertini theorem (see Lemma \ref{bertini-normal-finite-mor}) a general section $Y\,\subset\, X$
of ${\mathcal L}^m$ is normal and its pullback $W\,\subset\, Z$ in $f^*|{\mathcal L}^m|$ is also normal. Moreover,
$Y_{sing}\,\subset\, X_{sing}$. Since $\pi^N(X,\,P)$ (respectively, $\pi^N(Y,\,P_{|Y})$)
is an extension of the Galois group $G$ by the fundamental group scheme $\pi^N(Z)$ (respectively, $\pi^N(W)$).
Hence $\pi^N(Y,\,P{|Y})\,\longrightarrow\, 
\pi^N(X,\,P)$ is a surjection if $\dim X\, \geq\, 2$.
Moreover, if $X$ and $Z$ are smooth, it is an isomorphism provided $\dim X\,\geq\, 3$. 
\end{proof}

\begin{cor}
The above theorem is true with Nori fundamental group replaced by the \'etale fundamental group.
\end{cor}

\begin{proof}
Note that \'etale finite bundles restricted to a hypersurface are \'etale finite. Also if an
essentially finite bundle $V$ on $(X,\,P)$ restricts to an \'etale finite bundle on $(Y,\,P_{|Y})$, then
$V$ must be \'etale finite. Indeed, this is a consequence of the fact that $F$--trivial bundles on $(X,\,P)$ restrict
to $F$--trivial bundles on $(Y,\,P_{|Y})$ and the functor
$\Vect^f(X,\,P)\,\longrightarrow\,\Vect^f(Y,\,P_{|Y})$ given by restriction to $Y$ is fully faithful. 
\end{proof}

\section{Restriction of $\ell$-adic sheaves to curves}

Let $X$ be a normal projective variety over the algebraically closed base field $\mathbf k$ of characteristic $p
\,>\,0$, 
and let $X^o$ be an open subset contained in the smooth locus of $X$ such that $X\setminus X^o$ is the support of an 
effective Cartier divisor. Let $P$ be a branch data on $X$ such that $\BL(P)\cap X^o \,=\,\emptyset$.

A finite $\ell$--adic sheaf $\scrF$ on $X^o$ is a locally constant sheaf of ${\mathbb Z}_l$--modules $\Lambda$ with 
$\Lambda$ a finite set. Note that by \cite[Proposition 9.9]{formal.orbifolds} given any finite $\ell$--adic 
sheaf $\scrF$ on $X^o$, there exists a geometric branch data $P$ on $X$ with $\BL(P)\cap X^o\,=\,\emptyset$ such 
that $\scrF$ extends to an $\ell$--adic sheaf on $(X,P)$ (i.e., the continuous representation of 
$\pi_1^{et}(X^o,\,x)$ corresponding to $\scrF$ factors through $\pi_1^{et}((X,\,P),\,x)$).

Let $\scrF$ be a lisse $\overline{\Q}_{\ell}$--sheaf of rank $r$ on $X^o$. This is equivalent to a continuous 
representation $\rho\,:\,\pi_1^{et}(X^o,\,x)\,\longrightarrow\, {\rm GL}_r(K)$ where $K/Q_{\ell}$ is a finite 
extension.

\begin{definition}
A lisse $\overline{\Q}_{\ell}$--sheaf $\scrF$ on $X^o\,\subset\, X$ is said to be bounded by a geometric
branch data $P$ on $X$ if for some (and consequently all) $f\,:\,(Z,\,O)\,\longrightarrow\,(X,\,P)$ \'etale
cover the sheaf $f^*\scrF$ on $Z^o\,=\,f^{-1}(X^o)$ is tame (i.e., ${\rm Sw}(f^*\scrF)\,=\,0$;
see \cite[Section 3]{EK12}). This is equivalent to saying that the ramification of $\scrF$ is
bounded by $f$ in the sense of Drinfeld \cite{Dri12}.
\end{definition}

\begin{pro}
Let $(Z,\,O)\,\longrightarrow\, (X,\,P)$ be an \'etale cover.
Let $\scrF$ be a rank $r$ $\ell$--adic sheaf on $X^o$ bounded by $P$. Then the ramification of $\scrF$ is
bounded (in the sense of Deligne \cite[Definition 3.3]{EK12}) by the effective Cartier divisor $rD$ on $X$
where $D$ is an effective Cartier divisor supported on $X\setminus X^o$ such that $\cO_X(-D)$
is contained in the discriminant ideal $I(D_{Z/X})$.
\end{pro}

\begin{proof}
Note that the pull-back of $\scrF$ to $Z$ is tamely ramified. Then by the proof
of \cite[Theorem 3.9]{EK12}, ramification of $\scrF$ is bounded by $rD$.
\end{proof}

\begin{thm}
Let $X^o$ be a smooth quasi-projective variety of dimension $n$.
Let $\scrF$ be an irreducible finite $\ell$--adic sheaf on $X^o$ which extends to a $(n-1)$--separable
geometric projective formal orbifold $(X,\,P)$. Then there exists a smooth curve $C^o$ in $X^o$ such
that $\scrF_{|C^o}$ is irreducible.
\end{thm}

\begin{proof}
Let $x$ be any geometric point of $X^o$. Since $\scrF$ is an irreducible $\ell$--adic on $X^o$, it corresponds 
to an irreducible representation $\rho\,:\,\pi_1^{et}(X^o,x)\,\longrightarrow\, {\rm Aut}(\scrF_x)$. Since the
ramification of $\scrF$ 
is bounded by $P$, $\rho$ factors through $\rho_1\,:\,\pi_1(X,P)\,\longrightarrow\,{\rm Aut}(\scrF_x)$. Now we use
Theorem \ref{lefschetz-thm} successively to obtain a curve $C\,\subset\, X$ and a branch data $Q$ on $C$ such that 
$\BL(Q)\,\subset\,\BL(P)$ and there is surjection $\pi_1(C,\,Q)\,\longrightarrow\, \pi_1(X,\,P)$. Let
$C^o\,=\,C\setminus \BL(C)$, then 
$C^o\,\subset\, X^o$ and there is an epimorphism $\pi_1(C^o)\,\longrightarrow\, \pi_1(C,\,Q)$. Hence $\rho_1$
induces an irreducible representation of $\pi_1(C^o)$ which corresponds to $\scrF_{|C^o}$.
\end{proof}

\begin{thm}
Let $X^o$ be a smooth variety contained as an open subset in a projective normal variety $X$ of dimension $n$, 
and let $P$ be a branch data on $X$ such that $\BL(P)\cap X^o\,=\,\emptyset$. Also assume that $(X,\,P)$ is
super-separable 
(see Definition \ref{super-separable}) geometric projective formal orbifold. There exists a smooth curve 
$C^o\,\subset\,X^o$ such that for any irreducible lisse $\overline{\Q}_{\ell}$--sheaf $\scrF$ on $X^o$ bounded by 
a branch data $P$, the restriction $\scrF_{|C^o}$ is irreducible.
\end{thm}

\begin{proof}
Let $g\,:\,(Z_1,\,O)\,\longrightarrow\, (X,\,P)$ be an \'etale $G$--Galois cover and $Z_1^o\,=\,g^{-1}(X^o)$. 
 
Let $Z_2^o\,\longrightarrow\, Z_1^o$ be a finite irreducible \'etale cover tamely ramified along the divisor
$Z_1\setminus Z_1^o$. Let $f\,:\,Z_2\,\longrightarrow\, X$ be the normalization of $X$ in ${\mathbf k}(Z_2^o)$. 

By applying Lemma \ref{bertini-divisor} we obtain a smooth connected curve $C\,\subset\,X$ such that the 
curve $C_1\,=\,g^{-1}(C)$ is smooth and connected. Moreover $C_1$ intersects $Z_1\setminus Z_1^o$ 
transversally. As $C_2\,\longrightarrow\,C_1$ is tamely ramified along this intersection, $C_2$ is also 
smooth. Note that by \cite[Theorem 2.1(A)]{fulton-lazarsfeld} $C_2\,=\,f^{-1}(C)\,=\,C_1\times_{Z_1}Z_2$ is 
also connected and hence it is irreducible. Since this is true for every cover $Z_2\,\longrightarrow\,Z_1$ 
which is \'etale over $Z_1^o$ and tamely ramified elsewhere, the natural map 
$q_1\,:\,\pi_1(C_1^o)\,\longrightarrow\, \pi_1^t(Z_1^o)$ is surjective.

Let $\scrF$ be given by the continuous irreducible representation $\rho\,:\,\pi_1(X^o)\,\longrightarrow\, 
{\rm GL}(V)$. Since the ramification of $\scrF$ is bounded by $P$, we conclude that $g^*\scrF$ is a tame lisse 
$\overline{\Q}_{\ell}$--sheaf on $Z_1^o$. Hence $g^*\scrF$ which is given by $\rho_{|\pi_1(Z_1^o)}$ factors 
through the epimorphism $\pi_1(Z_1^o,\,x_2)\,\longrightarrow\, \pi_1^t(Z_1^o,\,x_2)$. Let $\rho_1\,:\,
\pi_1^t(Z_1^o)\,\longrightarrow\, {\rm GL}(V)$ be the 
resulting homomorphism.
 
Let $q\,:\,\pi_1(C^o)\,\longrightarrow\, \pi_1(X^o)$ be the homomorphism induced from $C^o\,
\hookrightarrow\, X^o$, and
let $$\rho_C\,=\,\rho\circ q\,:\,\pi_1(C^o,\,x)\,\longrightarrow\, {\rm GL}(V)$$ be the composite homomorphism.
 Then $\rho_C$ restricted to $\pi_1(C_1^o)$ is equal to $\rho_1\circ q_1$. Hence 
\begin{equation}\label{eq:inc}
\rho_C(\pi_1(C^o))\,\supset\,\rho_C(\pi_1(C_1^o))\,=\,\rho_1(\pi_1^t(Z_1^o))\,=\,\rho(\pi_1(Z_1^o)).
\end{equation}

Since $C_1^o\,\longrightarrow\, C^o$ is a connected $G$-Galois \'etale cover obtained by pullback of the
connected $G$--Galois \'etale cover $Z_1^o\,\longrightarrow\, X^o$, the homomorphism $q$ induces an isomorphism 
$$\pi_1(C^o)/\pi_1(C_1^o)\,\longrightarrow\, \pi_1(X^o)/\pi_1(Z_1^o).$$ Hence the image of $\rho_C(\pi_1(C^o,\,x))$ in 
$\rho(\pi_1(X^o))/\rho(\pi_1(Z_1^o))$ is the whole group. Therefore, we obtain that
$\rho_C(\pi_1(C^o))\,=\,\rho(\pi_1(X^o))$ because $\rho_C(\pi_1(C^o))\,\supset \,
\rho(\pi_1(Z_1^o))$ by equation \eqref{eq:inc}. Hence $\rho_C$ is also an irreducible representation.
\end{proof}

\section{Acknowledgements}

We are very grateful to the referee for pointing out an error in the proof of Theorem \ref{thmnl}
in an earlier version. The first-named author is partially supported by
a J. C. Bose Fellowship (JBR/2023/000003).

\section*{}

On behalf of all authors, the corresponding author states that there is no conflict of interest.

\end{document}